\documentclass[12pt]{amsart}

\usepackage{amsmath}
\usepackage{amssymb}
\usepackage{ascmac}
\usepackage{amsthm}

\makeatletter

\@addtoreset{equation}{section}
\makeatother 

\theoremstyle{plain}
\newtheorem{thm}{Theorem}[section]
\newtheorem*{thm*}{Theorem}
\newtheorem{prop}[thm]{Proposition}
\newtheorem{lem}[thm]{Lemma}
\newtheorem{cor}[thm]{Corollary}

\theoremstyle{definition}

\theoremstyle{remark}
\newtheorem{rem}{Remark}[section]

\newcommand{\ric}{\operatorname{Ric}}

\newcommand{\DS}{\operatorname{BSep}}
\newcommand{\bm}{\partial M}

\usepackage{latexsym}

\title[Concentration of eigenfunctions of the Laplacian]{Concentration
of eigenfunctions of the Laplacian on a closed Riemannian manifold}

\author{Kei Funano}
\address{Division of Mathematics \& Research Center for Pure and Applied Mathematics, Graduate School of Information Sciences, Tohoku University, 6-3-09 Aramaki-Aza-Aoba, Aoba-ku, Sendai 980-8579, Japan}
\email{kfunano@tohoku.ac.jp}

\author{Yohei Sakurai}
\address{Advanced Institute for Materials Research,Tohoku University, 2-1-1 Katahira, Aoba-ku, Sendai, 980-8577, Japan}
\email{yohei.sakurai.e2@tohoku.ac.jp}

\date{October 13, 2018}

\begin{document}
\maketitle

\begin{abstract}
We study concentration phenomena of eigenfunctions of
 the Laplacian on closed Riemannian manifolds. We prove that the volume measure of a closed manifold concentrates around nodal sets of eigenfunctions
 exponentially. Applying the method of Colding and Minicozzi we also
 prove restricted exponential concentration inequalities and restricted Sogge-type $L_p$ moment estimates of eigenfunctions. 
\end{abstract}

\section{Introduction}

 Eigenfunctions of the Laplacian naturally appeared as an important
 object in analysis and geometry (\cite{Z}). Their
 global behavior was vastly investigated in several literature. In this
 paper we study the global feature of eigenfunctions with focus on their concentration properties. 

 Let $(M,g)$ be a closed Riemannian manifold and $\varphi_{\lambda}$ be
 an eigenfunction corresponding to an eigenvalue $\lambda$ of the
 Laplacian on $M$.
 Br\"{u}ning \cite[Proposition 1]{Br} proved that
 there exists some constant $C=C(M,g)$ depending only on $(M,g)$ such that
 the $(C/\sqrt{\lambda}$)-neighborhood of the nodal set $\varphi_{\lambda}^{-1}(0)$ covers the whole manifold $M$ (see also \cite[Theorem 4]{Br2}, \cite[Theorem 4.1]{Z}),
 whose primitive version has been established by Courant (cf. \cite[Chapter V\hspace{-.1em}I, Section 6]{CH}). 
 One might wonder how much the measure of $M$ concentrates around the
 $(r/\sqrt{\lambda})$-neighborhood of $\varphi_{\lambda}^{-1}(0)$ for
 any given $r>0$. One of our main results answers this question. For
 $r>0$ and $\Omega \subset M$, let $B_{r}(\Omega)$ denote the closed $r$-neighborhood of $\Omega$.
We denote by $m_g$ the uniform volume measure on $M$ normalized as
 $m_g(M)=1$, i.e., $m_{g}:= v_g /v_{g}(M)$, where $v_g$ is the
 metric volume measure induced from $g$.

\begin{thm}\label{thm:nodal concentration}
Let $M$ be a closed Riemannian manifold. Then for all $r>0$ we have
\begin{equation*}
m_g(M\setminus B_{r}(\varphi_{\lambda}^{-1}(0))) \leq \exp \bigl(1-\sqrt{\lambda}\,r\bigl).
\end{equation*}
\end{thm}

In \cite{JM} Jakobson and Mangoubi gave upper and lower bounds for the volumes of
tubular neighborhoods at ``subwavelength" scales in the real analytic
setting. Recently Georgiev and Mukherjee \cite{GMu} estimated the volumes of
tubular neighborhoods at ``subwavelength" scales in the smooth
setting. In Remark \ref{comp} we compare Theorem \ref{thm:nodal
concentration} with these two results and with the Br\"{u}ning result.

Let $\ric_g$ denote the infimum of the Ricci curvature over $M$.
Under a lower Ricci curvature bound,
we also consider how much the eigenfunction $\varphi_{\lambda}$
concentrates to zero and obtain the following exponential concentration
inequality on large subsets:
\begin{thm}\label{thm:eigenfunction concentration}
Let $M$ be an $n$-dimensional closed Riemannian manifold with $\ric_g
 \geq -(n-1)$. Then there exists a constant $\mathcal{C}_{n}>0$ depending only on $n$ such that
the following holds:
If the eigenvalue $\lambda$ is at least $ \mathcal{C}_{n}$,
then for every $\xi \in (0,1)$,
there is a Borel subset $\Omega=\Omega_{\varphi_{\lambda},\xi} \subset M$ with $m_g(\Omega) \geq 1-\xi$ such that
\begin{equation}\label{eq:eigenfunction concentration}
m_g\bigl(\Omega \cap \{\vert \varphi_{\lambda} \vert >r \}\bigl)\leq \exp \left( 1-\frac{C_{n}\, \sqrt{\xi}\, }{\Vert \varphi_{\lambda} \Vert_{2}}\,r \right)
\end{equation}
for every $r>0$,
where $C_{n}>0$ is a constant depending only on $n$,
and $\Vert \cdot \Vert_{2}$ denotes the standard $L_{2}$ norm on $(M,m_g)$.
\end{thm}

\begin{rem}
The inequality (\ref{eq:eigenfunction concentration}) is meaningful for large $r>0$.
In fact,
if $r\leq \Vert \varphi_{\lambda} \Vert_{2} (C_{n} \sqrt{\xi})^{-1}$,
then (\ref{eq:eigenfunction concentration}) is trivial since the left hand side is at most $1$.
The Chebyshev inequality implies that (without the curvature assumption and the largeness of $\lambda$)
for every $r>0$,
\begin{equation*}
m_g\bigl(\{\vert \varphi_{\lambda} \vert >r \}\bigl)\leq \frac{\Vert \varphi_{\lambda} \Vert^{2}_{2}}{r^{2}}.
\end{equation*}
The inequality (\ref{eq:eigenfunction concentration}) gives a refinement of this inequality with respect to the decay order of the upper bound as $r\to \infty$ on a large subset $\Omega$.
\end{rem}

From the above theorem the standard argument yields the following.
\begin{cor}\label{cor:lp concentration}
Under the same setting and notation of Theorem \ref{thm:eigenfunction concentration},
there exists a constant $\mathcal{C}_{n}>0$ such that the following holds:
If $\lambda \geq \mathcal{C}_{n}$,
then for every $\xi \in (0,1)$,
there is a Borel subset $\Omega=\Omega_{\varphi_{\lambda},\xi} \subset M$ with $m_g(\Omega) \geq 1-\xi$ such that
\begin{equation}\label{eq:lp concentration}
\left(\int_{\Omega}\,\vert \varphi_{\lambda} \vert^{p}\,dm_g \right)^{\frac{1}{p}} \leq e\,\Gamma(p+1)^{\frac{1}{p}}\,\frac{\Vert \varphi_{\lambda}\Vert_{2}}{C_{n}\,\sqrt{\xi}}
\end{equation}
for any $p\geq 1$,
where $\Gamma$ is the gamma function.
\end{cor}

\begin{rem}
The Stirling formula tells us that
the ratio of the right hand side of (\ref{eq:lp concentration}) to $p \,\Vert \varphi_{\lambda}\Vert_{2}\,(C_{n}\,\sqrt{\xi})^{-1}$ tends to $1$ as $p\to \infty$.
\end{rem}

Let us mention the result due to Sogge. 
In \cite{So} (see also \cite[Theorem 9.2]{Z}) Sogge obtained a global $L_p$ moment estimate
for eigenfunctions: $\| \varphi_{\lambda}\|_p \leq
O(\lambda^{\delta(n,p)})\| \varphi_{\lambda}\|_{2}$, where
$\delta(n,p)$ is a function of $n$ and $p$. This inequality is
sharp for the round sphere $\mathbb{S}^n$. The crucial
point of Corollary \ref{cor:lp concentration} is that once we restrict
the eigenfunction $\varphi_{\lambda}$ on some large subset $\Omega$
then we do not need the $\lambda$-term of the Sogge inequality to bound the $L_p$ moment of
$\varphi_{\lambda}1_{\Omega}$ in terms of $\| \varphi_{\lambda}\|_2$.

\section{Dirichlet eigenvalues}\label{sec:Dirichlet eigenvalues}
Throughout this section,
let $(M,g)$ be a connected compact Riemannian manifold with boundary.
We denote by $\bm$ its boundary and by $d_g$ the Riemannian distance.

\subsection{Key estimates}
The key ingredient of the proof of our main results is the following
concentration inequality around the boundary in terms of the first
Dirichlet eigenvalue of the Laplacian. In the proof we closely
follows the argument of Gromov and Milman (\cite{GM}, \cite[Theorem
3.1]{L}). Their context was the first nontrivial eigenvalue on a closed manifold, and Neumann eigenvalue.
\begin{prop}\label{prop:exponential decay estimate}
For every $r>0$ we have
\begin{equation}\label{eq:exponential decay estimate}
m_g(M\setminus B_{r}(\bm)) \leq \exp \bigl(1-\sqrt{\lambda^{D}_{1}(M)}\,r \bigl),
\end{equation}
where $\lambda^{D}_{1}(M)$ is the first Dirichlet eigenvalue of the Laplacian on $M$.
\end{prop}
\begin{proof}
First,
we show that
for all $\epsilon,\,r>0$
we have
\begin{equation}\label{eq:relative volume comparison and eigenvalue}
(1+\epsilon^{2}\,\lambda^{D}_{1}(M))\,m_g(M\setminus B_{r+\epsilon}(\bm)) \leq m_g(M\setminus B_{r}(\bm)).
\end{equation}
We set $\Omega_{1}:=B_{r}(\bm),\,\Omega_{2}:=M\setminus B_{r+\epsilon}(\bm)$,
and put $v_{\alpha}:=m_g(\Omega_{\alpha})$ for each $\alpha =1,2$.
Let us define a Lipschitz function $\varphi:M\to \mathbb{R}$ by
\begin{equation*}
\varphi(x):=\min \left\{ \frac{1}{\epsilon} \,\,d_g(x,\Omega_{1}),\,\,1  \right\}.
\end{equation*}
The function $\varphi$ satisfies the following properties:
\begin{enumerate}\setlength{\itemsep}{+1.5mm}
\item $\varphi \equiv 0$ on $\Omega_{1}$, and $\varphi \equiv 1$ on $\Omega_{2}$;
\item $\Vert \nabla \varphi \Vert \leq \epsilon^{-1}$ $m_g$-almost everywhere on $M$,
\end{enumerate}
where $\nabla \varphi$ denotes the gradient of $\varphi$,
and $\Vert \cdot \Vert$ denotes the canonical norm induced from the Riemannian metric on $M$.
It follows that
\begin{equation*}
\int_{M}\,\varphi^{2}\,dm_g \geq v_{2},\quad \int_{M}\,\Vert \nabla \varphi \Vert^{2}\,dm_g \leq \frac{1}{\epsilon^{2}}\,(1-v_{1}-v_{2}).
\end{equation*}
The min-max principle leads us to
\begin{equation*}
\lambda^{D}_{1}(M)\leq \frac{\int_{M}\, \Vert \nabla \varphi \Vert^{2}\,d\,m_g}{\int_{M}\, \varphi^{2}\,d\,m_g} \leq \frac{1}{\epsilon^{2}\,v_{2}}\,(1-v_{1}-v_{2}).
\end{equation*}
This implies (\ref{eq:relative volume comparison and eigenvalue}).

Now,
let us prove (\ref{eq:exponential decay estimate}).
We put $\epsilon_{0}:=\lambda^{D}_{1}(M)^{-\frac{1}{2}}$.
We first consider the case where $r \in (0,\epsilon_{0})$.
Let $l \geq 1$ denote the integer determined by $\epsilon_{0}\,r^{-1} \in [(l+1)^{-1},l^{-1})$.
Using the inequality (\ref{eq:relative volume comparison and eigenvalue}) $l$ times,
we arrive at
\begin{align*}
m_g(M\setminus B_{r}(\bm)) &\leq m_g(M\setminus B_{l\,\epsilon_{0}}(\bm))\\
                                                    &\leq 2^{-1}\,m_g(M\setminus B_{(l-1)\,\epsilon_{0}}(\bm))\leq 2^{-l} \leq 2^{1-\frac{r}{\epsilon_{0}}}.
\end{align*}
This proves (\ref{eq:exponential decay estimate}).
In the case where $r \in [\epsilon_{0},\infty)$,
it holds that
\begin{equation*}
m_g(M\setminus B_{r}(\bm)) \leq \exp \left(1-\frac{r}{\epsilon_{0}}\right)
\end{equation*}
since the right hand side is at least $1$.
Therefore,
we conclude (\ref{eq:exponential decay estimate}).
\end{proof}

\subsection{Boundary separation distances}
We call $X=(X,d_{X},\mu_{X})$ a \textit{metric measure space with boundary}
when $X$ is a connected complete Riemannian manifold with boundary,
$d_{X}$ is the Riemannian distance,
and $\mu_{X}$ is a Borel probability measure on $X$.
Let $X=(X,d_{X},\mu_{X})$ be a metric measure space with boundary,
and let $k\geq 1$ be an integer.
For positive numbers $\eta_{1},\dots,\eta_{k}>0$,
we denote by $\mathcal{S}_{X}(\eta_{1},\dots,\eta_{k})$ the set of all sequences $\{\Omega_{\alpha}\}^{k}_{\alpha=1}$ of Borel subsets $\Omega_{\alpha}$ with $\mu_{X}(\Omega_{\alpha})\geq \eta_{\alpha}$.
For a sequence $\{\Omega_{\alpha}\}^{k}_{\alpha=1} \in \mathcal{S}_{X}(\eta_{1},\dots,\eta_{k})$,
we define
\begin{equation*}
\mathcal{D}_{X}\bigl(\{\Omega_{\alpha}\}^{k}_{\alpha=1}\bigl):=\min \left \{\, \min_{\alpha \neq \beta} d_{X}(\Omega_{\alpha},\Omega_{\beta}),\,\, \min_{\alpha} d_{X}(\Omega_{\alpha},\partial X)  \,\right\}.
\end{equation*}
The author \cite{S} has introduced the \textit{$(\eta_{1},\dots,\eta_{k})$-boundary separation distance $\DS(X;\eta_{1},\dots,\eta_{k})$ of $X$} as follows (see Definition 3.2 in \cite{S}):
If $\mathcal{S}_{X}(\eta_{1},\dots,\eta_{k}) \neq \emptyset$,
then 
\begin{equation*}
\DS(X;\eta_{1},\dots,\eta_{k}):=\sup\,\mathcal{D}_{X}\bigl(\{\Omega_{\alpha}\}^{k}_{\alpha=1}\bigl),
\end{equation*}
where the supremum is taken over all $\{\Omega_{\alpha}\}^{k}_{\alpha=1} \in \mathcal{S}_{X}(\eta_{1},\dots,\eta_{k})$;
otherwise,
$\DS(X;\eta_{1},\dots,\eta_{k}):=0$.
The second author \cite{S} has presented the following relation with the
Dirichlet eigenvalue:
\begin{lem}[{\cite[Lemma 4.1]{S}}]\label{lem:Dirichlet eigenvalue and Dirichlet separation distance}
For all $\eta_{1},\dots,\eta_{k}>0$,
we have
\begin{equation*}
\DS\left((M,d_g,m_g);\eta_{1},\dots,\eta_{k}\right)\leq \frac{2}{\sqrt{\lambda^{D}_{k}(M)\, \min_{\alpha=1,\dots,k} \eta_{\alpha}}},
\end{equation*}
where $\lambda^{D}_{k}(M)$ is the $k$-th Dirichlet eigenvalue of the Laplacian on $M$.
\end{lem}

By applying Proposition \ref{prop:exponential decay estimate} to our setting,
we obtain the following refined estimate in the case where $k=1$:
\begin{thm}\label{thm:logarithmic Dirichlet eigenvalue and Dirichlet separation distance}
For every $\eta>0$,
we have
\begin{equation*}
\DS \left((M,d_g,m_g);\eta\right)\leq \frac{1}{\sqrt{\lambda^{D}_{1}(M)}} \log \frac{e}{\eta}.
\end{equation*}
\end{thm}
\begin{proof}
We may assume that
the left hand side is positive.
Fix a Borel subset $\Omega\subset M$ with $m_g(\Omega)\geq \eta$.
From Proposition \ref{prop:exponential decay estimate},
for every $r>0$ with $r>\lambda^{D}_{1}(M)^{-\frac{1}{2}}(1-\log \eta)$,
we derive $m_g(B_{r}(\bm)) > 1-\eta$;
in particular,
$B_{r}(\bm)\cap \Omega \neq \emptyset$ and $d_g(\Omega,\bm)\leq r$.
By letting $r\to \lambda^{D}_{1}(M)^{-\frac{1}{2}}(1-\log \eta)$,
we obtain $d_g(\Omega,\bm) \leq \lambda^{D}_{1}(M)^{-\frac{1}{2}}(1-\log \eta)$.
Since $\Omega$ is arbitrary,
we complete the proof.
\end{proof}

\section{Proof of the main results}\label{sec:Proof of the main results}
In this section,
we will prove the main results.
In what follows,
let $M$ be an $n$-dimensional closed Riemannian manifold and
$\varphi_{\lambda}$ be an eigenfunction corresponding to an eigenvalue $\lambda$ of the Laplacian on $M$.
For a Borel subset $\Omega\subset M$,
let $m_{\Omega}$ stand for the normalized volume measure on $\Omega$ defined as
\begin{equation}\label{eq:normalized measure}
m_{\Omega}:=\frac{1}{v_g(\Omega)}\,v_g|_{\Omega},
\end{equation}where $v_g$ is the volume measure of $M$.

\subsection{Proof of Theorem \ref{thm:nodal concentration}}
Let us prove Theorem \ref{thm:nodal concentration}.
\begin{proof}[Proof of Theorem \ref{thm:nodal concentration}]
We fix a nodal domain $\Omega$ of $\varphi_{\lambda}$, i.e., a connected
 component of $M\setminus \varphi_{\lambda}^{-1}(0)$.
Applying the same argument in the proof of Proposition \ref{prop:exponential decay estimate} to $\Omega$,
we see
\begin{equation*}
\frac{m_g(\Omega \setminus B_{r}(\partial \Omega))}{m_g(\Omega)} =    m_{\Omega}(\Omega \setminus B_{r}(\partial \Omega))\leq \exp\bigl(1-\sqrt{\lambda^{D}_{1}(\Omega)}\,r\bigl),
\end{equation*}
where $\partial \Omega$ is the boundary of $\Omega$,
and $\lambda^{D}_{1}(\Omega)$ is the first Dirichlet eigenvalue of the Laplacian on $\Omega$.
It is well-known that
$\lambda^{D}_{1}(\Omega)$ is equal to $\lambda$ (\cite[Lemma 1 in Chapter 1]{C}).
Hence,
\begin{equation}\label{eq:application of Dirichlet eigenvalue estimate}
m_g(\Omega \setminus B_{r}(\partial \Omega)) \leq \exp\bigl(1-\sqrt{\lambda}\,r\bigl)\,m_g(\Omega).
\end{equation}

We now decompose the set $M \setminus \varphi^{-1}_{\lambda}(0)$ into the nodal domains $\Omega_{\alpha}$ of $\varphi_{\lambda}$ as $M \setminus \varphi^{-1}_{\lambda}(0)=\sqcup_{\alpha}\, \Omega_{\alpha}$.
From (\ref{eq:application of Dirichlet eigenvalue estimate})
we derive
\begin{align*}
m_g(M\setminus B_{r}(\varphi^{-1}_{\lambda}(0))) &=     \sum_{\alpha}\,m_g(\Omega_{\alpha} \setminus B_{r}(\partial \Omega_{\alpha}))\\
                                                                 &\leq  \exp\bigl(1-\sqrt{\lambda}\,r\bigl)\,      \sum_{\alpha}\,m_g(\Omega_{\alpha}).
\end{align*}
By $\sum_{\alpha}\,m_g(\Omega_{\alpha})\leq 1$,
we obtain the desired inequality.
\end{proof}

\begin{rem}\label{comp}\upshape
 Let us compare our estimate of the volumes of tubular neighborhoods of
 the nodal sets 
 \begin{align}\label{our ineq}
  1-\exp \bigl(1-\sqrt{\lambda}\,r\bigl)\leq m_{g}(B_{r}(\varphi_{\lambda}^{-1}(0))) 
  \end{align}with the other known results. 

 First recall that the nodal set is $C/\sqrt{\lambda}$-dense for some
 constant $C=C(M,g)$. Thus the right-hand side of (\ref{our ineq}) is always
 $1$ whenever $r\geq C/\sqrt{\lambda}$. For $r\leq 1/\sqrt{\lambda}$ the
 left-hand side of (\ref{our ineq}) is at most zero. Thus the inequality
 (\ref{our ineq}) has meanings only in the range
 $1/\sqrt{\lambda}<r<C/\sqrt{\lambda}$ and is meaningless unless $C=C(M,g)>1$.
 There are several proofs for the $C/\sqrt{\lambda}$-denseness of the
 nodal set (see \cite{Br, CM, CM2, HL, Lo, M} and consult \cite{Z, Z1} for example), however to the author's
 knowledge, no explicit quantitative estimate for the constant $C=C(M,g)$ are known.
 In the case of the $n$-dimensional unit round sphere $\mathbb{S}^n$ the first nontrivial
 eigenvalue $n$ is given by restriction of $n+1$ coordinate
 projections and hence $C(\mathbb{S}^{n})\geq \pi \sqrt{n} /2>1$. 

 In \cite[Theorem 1.2]{JM} following the work of Donnelly and Fefferman \cite{DF} Jakobson and Mangoubi proved that if $(M,g)$ is real analytic, then
\begin{align}\label{tubu ineq}
C_1\sqrt{\lambda}r\leq v_g(B_r(\varphi_{\lambda}^{-1}(0)))\leq C_2\sqrt{\lambda}r
\end{align}provided that $r\leq C_3 /\sqrt{\lambda} $, where $C_1,C_2,C_3$
are constants depending on $(M,g)$,
and $v_{g}$ is the (unnormalized) metric volume measure induced from $g$.
 No explicit quantitative estimate for the constants
 $C_1,C_2,C_3$ were given in their proof (and hence we cannot compare
 the left-most inequality in (\ref{tubu ineq}) with our inequality (\ref{our ineq})). One reason for this is that they focused on
 local neighborhoods so that the metric can be developed in power series.

 In the smooth setting, modifying the work of Logunov \cite{Lo, Lo2} and
 Logunov and Malinnikova \cite{LM}, Georgiev and Mukherjee \cite[Theorem
 1.2]{GMu} proved the following: For any sufficiently small
 $\varepsilon>0$ it holds
 \begin{align*}
  C_1\lambda^{1/2-\varepsilon}r \leq
  v_g(B_r(\varphi_{\lambda}^{-1}(0)))\leq C_2 \lambda^{\kappa}r
  \end{align*}for any $r\in (0,r_0/\sqrt{\lambda})$, where
 $r_0=r_0(M,g)$, $C_1=C_1(M,g,\varepsilon)$, $C_2=C_2(M,g)$, and
 $\kappa=\kappa(M,g)$ are some constants. Again no explicit
 quantitative estimate for the constants
 $r_0,C_1,C_2$ were given in their proof but they remarked that their
 constant $C_1$ goes to zero as $\varepsilon$ tends to zero.
 \end{rem}

\subsection{Proof of Theorem \ref{thm:eigenfunction concentration}}
In the present subsection,
we will give a proof of Theorem \ref{thm:eigenfunction
concentration} by applying Theorem \ref{thm:nodal concentration}.
In order to apply the theorem we need to control the gradient of
an eigenfunction on large subsets on $M$. To do so we shall follow the
argument of Colding and Minicozzi in \cite{CM}.

We first recall the following (see the proof of \cite[Theorem 1.1]{CM}):
\begin{lem}[\cite{CM}]\label{lem:net of nodal set}
If $\ric_g \geq -(n-1)$,
then there is a constant $C_{n}>0$ depending only on $n$ such that $M=B_{R}(\varphi^{-1}_{\lambda}(0))$,
where $R:=C_{n}\,\lambda^{-\frac{1}{2}}$.
\end{lem}

In the proof of \cite[Theorem 1.1]{CM} Colding and Minicozzi have also obtained the following fact by
combining the mean value inequality with the Bochner formula:
\begin{lem}[\cite{CM}]\label{lem:gradient estimate}
If $\ric_g \geq -(n-1)$,
then there exists a constant $C_{n}>0$ depending only on $n$ such that for all $x \in M$ and $r \in (0,1]$,
the supremum of $\Vert \nabla \varphi_{\lambda} \Vert^{2}$ over $B_{r}(x)$ is at most
\begin{equation*}
\exp\left(C_{n}\,\bigl(1+r \,\sqrt{2(\lambda+n-1)}\bigl)\right)\,\int_{B_{2r}(x)}\,\Vert \nabla \varphi_{\lambda} \Vert^{2}\,dm_{B_{2r}(x)},
\end{equation*}
where $m_{B_{2r}(x)}$ is the normalized measure on $B_{2r}(x)$ defined as $(\ref{eq:normalized measure})$.
\end{lem}

In order to prove Theorem \ref{thm:eigenfunction concentration},
we show the following assertion based on the above two lemmas.
 The idea goes back to Colding and Minicozzi (\cite[Theorem 1.1]{CM}).
\begin{lem}\label{lem:neighborhood of nodal set}
We assume $\ric_g \geq -(n-1)$.
Then there exists a constant $\mathcal{C}_{n}>0$ depending only on $n$ such that
the following holds:
If $\lambda \geq \mathcal{C}_{n}$,
then for every $\xi \in (0,1)$,
there exists a Borel subset $\Omega=\Omega_{\varphi_{\lambda},\xi} \subset M$ with $m_g(\Omega) \geq 1-\xi$ such that
\begin{equation}\label{eq:neighborhood of nodal set}
\Omega \cap B_{r}(\varphi^{-1}_{\lambda}(0)) \subset \Omega \cap \left\{ \vert \varphi_{\lambda} \vert \leq C_{n}\, \sqrt{\frac{\lambda}{\xi}}\,   \Vert \varphi_{\lambda}\Vert_{2}  \,r  \right\}
\end{equation}
for every $r>0$,
where $C_{n}>0$ is a constant depending only on $n$.
\end{lem}
\begin{proof}
By Lemma \ref{lem:net of nodal set},
there exists a constant $C_{1,n}>0$ depending only on $n$ such that $M=B_{R}(\varphi^{-1}_{\lambda}(0))$,
where $R:=C_{1,n}\,\lambda^{-\frac{1}{2}}$.
We define $\mathcal{C}_{n}:=\max\{(10\,C_{1,n})^{2},n-1\}$,
and suppose $\lambda\geq \mathcal{C}_{n}$.
Then $10R \in (0,1]$,
and hence
Lemma \ref{lem:gradient estimate} implies that
there is a constant $C_{2,n}>0$ depending only on $n$ such that for every $x \in M$,
the supremum of $\Vert \nabla \varphi_{\lambda} \Vert^{2}$ over $B_{10R}(x)$ is smaller than or equal to
\begin{equation*}
\exp \left(C_{2,n}\,\bigl(1+10R\,\sqrt{2(\lambda+n-1)}\bigl)\right)\int_{B_{20R}(x)} \, \Vert \nabla \varphi_{\lambda} \Vert^{2} \, dm_{B_{20R}(x)}.
\end{equation*}
From $\lambda\geq n-1$,
it follows that
for every $x\in M$
we have
\begin{equation}\label{eq:application of gradient estimate}
\sup_{B_{10R}(x)}\,\Vert \nabla \varphi_{\lambda} \Vert^{2}\leq C_{3,n}\,\int_{B_{20R}(x)} \, \Vert \nabla \varphi_{\lambda} \Vert^{2} \, dm_{B_{20R}(x)},
\end{equation}
where $C_{3,n}:=\exp (C_{2,n}\,(1+20\, C_{1,n}))$.

Let us take a maximal family $\{B_{R}(x_{i})\}_{i \in I}$ of disjoint balls centered at $\varphi^{-1}_{\lambda}(0)$.
The maximality leads to $\varphi^{-1}_{\lambda}(0)\subset \bigcup_{i\in I}B_{2R}(x_{i})$;
in particular,
\begin{equation}\label{eq:maximality}
M=B_{R}(\varphi^{-1}_{\lambda}(0))=\bigcup_{i\in I}B_{3R}(x_{i}).
\end{equation}
By the Bishop-Gromov volume comparison,
there is a constant $C_{4,n}>0$ depending only on $n$ such that
the multiplicity of $\{B_{20\,R}(x_{i})\}_{i \in I}$ is at most $C_{4,n}$.
Furthermore,
the standard covering argument tells us that
\begin{align}\label{eq:covering estimate}
\sum_{i\in I} \,\int_{B_{20R}(x_{i})}\,\Vert \nabla \varphi_{\lambda}\Vert^{2}\,dv_g&\leq C_{4,n}\int_{M}\,\Vert \nabla \varphi_{\lambda}\Vert^{2}\,dv_g\\ \notag
                                                                                                                                  &  =  C_{4,n}\,\lambda\,\int_{M}\, \varphi_{\lambda}^{2}\,dv_g.
\end{align}
We define $J\subset I$ by the set of all $i\in I$ satisfying
\begin{equation}\label{eq:integral bound}
\int_{B_{20R}(x_{i})}\,\Vert \nabla \varphi_{\lambda}  \Vert^{2}\,dm_{B_{20R}(x_{i})} \leq C_{4,n}\,\frac{\lambda}{\xi}\,\int_{M}\,\varphi^{2}_{\lambda}\, dm_g,
\end{equation}
and $J':=I\setminus J$.
Note that
if $i\in J$,
then we deduce
\begin{equation}\label{eq:gradient bound}
\sup_{B_{10R}(x_{i})}\,\Vert \nabla \varphi_{\lambda} \Vert^{2}\leq C_{5,n}\,\Lambda
\end{equation}
from (\ref{eq:application of gradient estimate}) and (\ref{eq:integral bound}),
where $C_{5,n}:=(C_{3,n}\,C_{4,n})^{\frac{1}{2}}$ and $\Lambda:=(\lambda\,\xi^{-1})^{\frac{1}{2}}\,\Vert \varphi_{\lambda}\Vert_{2}$.
We set
\begin{equation*}
\Omega:=\bigcup_{i \in J}\,B_{3R}(x_{i}),\quad \Omega':=\bigcup_{i \in J'}\,B_{3R}(x_{i}).
\end{equation*}
We will verify that
$\Omega$ is a desired Borel subset.
By the definition of $J'$,
\begin{equation*}
C_{4,n}\,\frac{\lambda}{\xi}\, \int_{M}\,\varphi_{\lambda}^{2}\,dv_g\,\sum_{i \in J'}\,m_g(B_{20R}(x_{i})) \leq \sum_{i \in J'}\int_{B_{20R}(x_{i})}\,\Vert \nabla \varphi_{\lambda}\Vert^{2} \,dv_g.
\end{equation*}
This together with (\ref{eq:covering estimate}) implies $\sum_{i \in J'}\,m_g(B_{20R}(x_{i}))\leq \xi$,
and hence $m_g(\Omega')\leq \xi$.
On the other hand,
$M=\Omega \cup \Omega'$ since (\ref{eq:maximality}).
Therefore,
one can conclude $m_g(\Omega)\geq 1-\xi$.

Now,
we check the inclusion (\ref{eq:neighborhood of nodal set}) for $C_{n}=C_{5,n}$.
For $r>0$,
we fix $x\in B_{r}(\varphi^{-1}_{\lambda}(0))\cap \Omega$.
Then the following hold:
\begin{enumerate}\setlength{\itemsep}{+1.5mm}
\item there exists $x_{0} \in M$ with $\varphi_{\lambda}(x_{0})=0$ such that $d_g(x,x_{0})\leq r$; \label{item:refined vanish on boundary}
\item there is $i_{0}\in J$ with $\varphi_{\lambda}(x_{i_{0}})=0$ such that $d_g(x,x_{i_{0}})\leq 3R$. \label{item:refined gradient estimate}
\end{enumerate}
Let us consider the case where $r \in (0,3R]$.
We take a minimal geodesic $\gamma:[0,d_g(x,x_{0})]\to M$ from $x$ to $x_{0}$.
Using the triangle inequality and $d_g(x,x_{0})\leq 3R$,
we see that
$\gamma$ lies in $B_{10R}(x_{i_{0}})$.
By $\varphi_{\lambda}(x_{0})=0$,
the Cauchy-Schwarz inequality,
$d_g(x,x_{0})\leq r$ and (\ref{eq:gradient bound}),
we obtain
\begin{align*}
\vert \varphi_{\lambda}(x) \vert &\leq \int^{d_g(x,x_{0})}_{0}\, \Vert \nabla \varphi_{\lambda} \Vert(\gamma(t)) \,dt \leq r\,\sup_{B_{10R}(x_{i_{0}})}\,\Vert \nabla \varphi_{\lambda} \Vert \leq C_{5,n}\,\Lambda\,r,
\end{align*}
and this proves (\ref{eq:neighborhood of nodal set}).
When $r \in (3R,\infty)$,
one can prove (\ref{eq:neighborhood of nodal set}) by taking a minimal geodesic from $x$ to $x_{i_{0}}$,
and by a similar argument to that in the case where $r \in (0,3R]$.
We complete the proof.
\end{proof}

We are now in a position to prove Theorem \ref{thm:eigenfunction concentration}.
\begin{proof}[Proof of Theorem \ref{thm:eigenfunction concentration}]
We assume $\ric_g \geq -(n-1)$.
By Lemma \ref{lem:neighborhood of nodal set},
there is a constant $\mathcal{C}_{n}>0$ depending only on $n$ such that
the following holds:
If we have $\lambda \geq \mathcal{C}_{n}$,
then for every $\xi \in (0,1)$,
there exists a Borel subset $\Omega=\Omega_{\varphi_{\lambda},\xi} \subset M$ with $m_g(\Omega) \geq 1-\xi$ such that
for every $r>0$
\begin{equation}\label{eq:second neighborhood of nodal set}
\Omega \cap B_{r}(\varphi^{-1}_{\lambda}(0)) \subset \Omega \cap \left\{ \vert \varphi_{\lambda} \vert \leq  C_{n}\, \sqrt{\frac{\lambda}{\xi}}\,   \Vert \varphi_{\lambda}\Vert_{2}  \,r  \right\},
\end{equation}
where $C_{n}>0$ is a constant depending only on $n$.

By (\ref{eq:second neighborhood of nodal set}),
for every $r>0$
we see
\begin{equation*}
\Omega \cap B_{(C_{n}\,\Lambda)^{-1}\,r}(\varphi^{-1}_{\lambda}(0)) \subset \Omega \cap \left\{ \vert \varphi_{\lambda} \vert \leq r  \right\},
\end{equation*}
where $\Lambda:=(\lambda\,\xi^{-1})^{\frac{1}{2}}\,\Vert \varphi_{\lambda}\Vert_{2}$.
It follows that
\begin{align*}
m_g\bigl(\Omega \cap \{\vert \varphi_{\lambda} \vert >r\}\bigl) &\leq m_g\bigl(\Omega \setminus B_{(C_{n}\,\Lambda)^{-1}\,r}(\varphi^{-1}_{\lambda}(0)) \bigl)\\
                                                                                          &\leq m_g\bigl(M \setminus B_{(C_{n}\,\Lambda)^{-1}\,r}(\varphi^{-1}_{\lambda}(0)) \bigl).
\end{align*}
Due to Theorem \ref{thm:nodal concentration},
we arrive at the desired inequality
\begin{align*}
m_g\bigl(\Omega \cap \{\vert \varphi_{\lambda} \vert >r\}\bigl) &\leq \exp\bigl(1-\sqrt{\lambda}\,(C_{n}\,\Lambda)^{-1}\,r\bigl)\\
                                                                                          &  =   \exp\left(1-\frac{C^{-1}_{n}\,\sqrt{\xi}}{\Vert \varphi_{\lambda} \Vert_{2} }\,r \right).
\end{align*}
This completes the proof of Theorem \ref{thm:eigenfunction concentration}.
\end{proof}

\subsection{Proof of Corollary \ref{cor:lp concentration}}
We finally prove Corollary \ref{cor:lp concentration}.
\begin{proof}[Proof of Corollary \ref{cor:lp concentration}]
Let $p\in[1,\infty)$.
We assume $\ric_g \geq -(n-1)$.
By Theorem \ref{thm:nodal concentration},
there exists a constant $\mathcal{C}_{n}>0$ depending only on $n$ such that
the following holds:
If $\lambda \geq \mathcal{C}_{n}$,
then for every $\xi \in (0,1)$,
there exists a Borel subset $\Omega=\Omega_{\varphi_{\lambda},\xi} \subset M$ with $m_g(\Omega) \geq 1-\xi$ such that
\begin{equation}\label{eq:main theorem2}
m_g\bigl(\Omega \cap \{\vert \varphi_{\lambda} \vert >r \}\bigl)\leq \exp \left( 1-\frac{C_{n}\, \sqrt{\xi}}{\Vert \varphi_{\lambda} \Vert_{2}}\,r \right)
\end{equation}
for every $r>0$,
where $C_{n}>0$ is a constant depending only on $n$.

The Cavalieri principle yields
\begin{equation*}\label{eq:basic equality}
\int_{\Omega}\,\vert \varphi_{\lambda} \vert^{p}\,d\,m_g=\int^{\infty}_{0}\,m_g(\Omega \cap \{ \vert \varphi_{\lambda} \vert^{p}>t    \} )\,dt.
\end{equation*}
By letting $r:=t^{\frac{1}{p}}$,
and by change of variables,
we have
\begin{equation*}
\int_{\Omega}\,\vert \varphi_{\lambda} \vert^{p}\,d\,m_g=p\,\int^{\infty}_{0}\,m_g(\Omega \cap \{ \vert \varphi_{\lambda} \vert>r    \} )\,r^{p-1}\,dr.
\end{equation*}
Using (\ref{eq:main theorem2}),
and change of variables again,
we arrive at
\begin{align*}
\int_{\Omega}\,\vert \varphi_{\lambda} \vert^{p}\,d\,m_g &\leq e\,p\,\int^{\infty}_{0}\, \exp \left( -\frac{C_{n} \sqrt{\xi}}{\Vert \varphi_{\lambda} \Vert_{2}}\,r \right)  \,r^{p-1}\,dr\\
                                                                                 &  =  e\,p\, \left(  \frac{\Vert \varphi_{\lambda} \Vert_{2}}{C_{n}\, \sqrt{\xi}} \right)^{p}\,   \int^{\infty}_{0}\, \exp \left( -s \right)  \,s^{p-1}\,ds,
\end{align*}
where $s:=(C_{n}\, \sqrt{\xi}\, \Vert \varphi_{\lambda} \Vert^{-1}_{2})\,r$.
Hence
we have
\begin{equation*}
\left(\int_{\Omega}\,\vert \varphi_{\lambda} \vert^{p}\,d\,m_g\right)^{\frac{1}{p}}
\leq e^{\frac{1}{p}}\,\bigl(p\,\Gamma(p)\bigl)^{\frac{1}{p}} \,\frac{\Vert \varphi_{\lambda} \Vert_{2}}{C_{n}\, \sqrt{\xi}}
\leq e\,\Gamma(p+1)^{\frac{1}{p}}\,  \frac{\Vert \varphi_{\lambda} \Vert_{2}}{C_{n}\, \sqrt{\xi}}.
\end{equation*}
Thus,
we complete the proof of Corollary \ref{cor:lp concentration}.
\end{proof}


\bigbreak
\noindent
{\it Acknowledgments.} The authors would like to thank the anonymous referee
and Professor Mayukh Mukherjee for their helpful and useful comments.

\end{document}